\documentclass[12pt,bezier]{article}
\usepackage{times}
\usepackage{booktabs}
\usepackage{pifont}
\usepackage{floatrow}
\usepackage{caption}
\usepackage{mathrsfs}
\usepackage[fleqn]{amsmath}
\usepackage{amsfonts,amsthm,amssymb,mathrsfs,bbding}
\usepackage{txfonts}
\usepackage{graphics,multicol}
\usepackage{graphicx}
\usepackage{color}
\usepackage{amssymb}
\usepackage{caption}
\usepackage{cite}
\usepackage{latexsym,bm}
\usepackage{indentfirst}
\usepackage{color}
\usepackage[colorlinks=true,anchorcolor=blue,filecolor=blue,linkcolor=blue,urlcolor=blue,citecolor=blue]{hyperref}
\usepackage{extarrows}
\usepackage{cite}
\usepackage{latexsym,bm}
\usepackage{mathtools}
\evensidemargin=\oddsidemargin\topmargin=-1.5cm

\newtheorem{thm}{Theorem}[section]
\newtheorem{case}{Case}

\newtheorem{prob}{Problem}[section]
\newtheorem{claim}{Claim}

\newtheorem{lemma}{Lemma}[section]
\newtheorem{cor}{Corollary}[section]

\newtheorem{observation}{Observation}[section]

\theoremstyle{definition}

\addtocounter{section}{0}

\begin{document}
\title{Spectral radius and rainbow $k$-factors in a bipartite graph family\footnote{Supported by the National Natural Science Foundation of China
{(No. 12371361)} and Distinguished Youth Foundation of Henan Province {(No. 242300421045)}.}}
\author{\bf  Meng Chen$^{a}$, {\bf Ruifang Liu$^{a}$}\thanks{Corresponding author.
E-mail addresses: cm4234635@163.com (M. Chen), rfliu@zzu.edu.cn (R. Liu).}\\
{\footnotesize $^a$ School of Mathematics and Statistics, Zhengzhou University, Zhengzhou, Henan 450001, China} }
\date{}

\maketitle
{\flushleft\large\bf Abstract}
Let $\mathcal{G}=\{G_1, G_2, \ldots , G_{kn}\}$ be a family of balanced bipartite graphs on the same vertex set $[2n]$. A rainbow $k$-factor of $\mathcal{G}$ is defined as a $k$-factor such that any two distinct edges come from different graphs in $\mathcal{G}.$ In this paper, we provide a tight sufficient condition in terms of the spectral radius for a family of balanced bipartite graphs $\mathcal{G}$ to contain a rainbow $k$-factor. Furthermore, we completely characterize the corresponding spectral extremal graph.

\begin{flushleft}
\end{flushleft}

\section{Introduction}
Throughout this paper, we only consider simple undirected graphs. Let $V(G)$ and $E(G)$ denote the vertex set and edge set of $G,$ respectively. For any vertex $u \in V(G),$ let $N_G(u)$ the neighborhood of $u$ in $G.$ The union of graphs $G_1$ and $G_2$ is denoted by $G_1\cup G_2.$ Let $G$ be a bipartite graph with the bipartition $(X, Y).$ We denote by $\widehat{G}$ its quasi-complement, where $V(\widehat{G})=V(G)$ and for any $x\in X$ and $y\in Y,$ $xy\in E(\widehat{G})$ if and only if $xy\notin E(G).$ Let $G_1$ and $G_2$ be two bipartite graphs with the bipartition $(X_1, Y_1)$ and $(X_2, Y_2),$ respectively. We use $G_1\sqcup G_2$ to denote the bipartite graph obtained from $G_1 \cup G_2$ by adding all possible edges between $X_1$ and $Y_2$ and all possible edges between $Y_1$ and $X_2.$ A bipartite graph with the bipartition $(X, Y)$ is called a balanced bipartite graph if $|X|=|Y|.$

The adjacency matrix of $G$ is the $|V(G)|\times|V(G)|$ matrix $A(G)=(a_{uv}),$ where $a_{uv}$ is the number of edges joining $u$ and $v.$ The eigenvalues of $G$ are defined as the eigenvalues of its adjacency matrix $A(G).$ The maximum modulus of eigenvalues of $A(G)$ is called the spectral radius of $G$ and denoted by $\rho(G).$

Over the past seven decades, factor theory has occupied a central role in graph theory. An $[a, b]$-factor of a graph $G$ is a spanning subgraph $F$ such that $a \leq d_F(v) \leq b$ for each $v \in V (G)$. If $a=b=k,$ then a $[a,b]$-factor is a $k$-factor. In particular, a perfect matching is a $1$-factor, and a Hamilton cycle is a connected $2$-factor. Motivated by the classical work of Tutte \cite{Tutte1947, Tutte1952} and Hall \cite{Hall1935}, extensive research has been conducted to characterize the structural properties which ensure the existence of $k$-factors in graphs \cite{Enomoto1988, Katerinis1987, Lu2020, Ore1959, Rado1949}. The existence problem of $k$-factors has also been extensively investigated under spectral radius conditions. In 2010, Fiedler and Nikiforov \cite{Fiedler2010} presented a spectral condition for the existence of a Hamilton cycle in a graph, and their result was subsequently improved in \cite{Chen2018, Ge2020}. In 2021, O \cite{O2021} provided a sufficient condition based on the spectral radius to guarantee the existence of a perfect
matching in a connected graph. By extending the above results to general graph factors, Cho et al. \cite{Cho2021} proposed a conjecture on a spectral radius condition for the existence of $[a,b]$-factors in graphs. Fan, Lin, and Lu \cite{Fan2022} proved the conjecture for $n\geq 3a+b-1.$ Wei and Zhang\cite{Wei2023} completely confirmed this conjecture. Hao and Li \cite{Hao2024}
strengthened the result of Wei and Zhang.

A bipartite graph $G$ contains a $k$-factor only if $G$ is a balanced bipartite graph. For balanced bipartite graphs, the existence of $k$-factors has also received considerable attention\cite{Bai2023, Chiba2017, Ore1957}. From the spectral perspective, Lu, Liu and Tian \cite{Lu2012} presented sufficient conditions in terms of the spectral radius for the existence of a Hamilton cycle in balanced bipartite graphs. Fan and Lin\cite{Fan2026} provided a spectral radius condition for a balanced bipartite graph to contain a $k$-factor, where $k\geq 2.$
\begin{thm}[Fan and Lin\cite{Fan2026}]\label{thm1.1}
Let $2\leq k\leq \frac{n}{2}-1,$ and let $G$ be a connected balanced bipartite graph with order $n.$ If
\begin{eqnarray*}
    \rho(G)\geq \rho(K_{k-1, n-1}\sqcup \widehat{K_{n-k+1, 1}}),
\end{eqnarray*}
then $G$ contains a $k$-factor unless $G\cong K_{k-1, n-1}\sqcup \widehat{K_{n-k+1, 1}}.$
\end{thm}

Let $\mathcal{G}= \{G_1, G_2,\ldots, G_k\}$ be a family of not necessarily distinct graphs with the same vertex
set $V.$ Let $H$ be a graph with $k$ edges on the vertex set $V(H)\subseteq V.$ We say that $\mathcal{G}$
contains a rainbow copy of $H$ if there exists a bijection $\phi: E(H) \rightarrow[k]$ such that $e \in E(G_{\phi(e)})$ for all $e\in E(H),$ where $[k]=\{1, 2, \ldots, k\}.$ In other words, each edge of $H$ comes from a different graph $G_i$. In particular, if
$G_1=G_2=\cdots=G_k$, then the rainbow copy of $H$ reduces to the classical copy of $H.$ The following general question was proposed by Joos and Kim in \cite{joos2020rainbow}.
\begin{prob}\label{prob1.1}
Let $H$ be a graph with $k$ edges and $\mathcal{G}= \{G_1, G_2,\ldots, G_k\}$ be a family of not necessarily distinct graphs on the same vertex set $V.$ Which properties and constraints imposed on the family $\mathcal{G}$ can yield a rainbow copy of $H$?
\end{prob}
Focusing on Problem \ref{prob1.1}, researchers have investigated whether classical results can be extended to their rainbow versions. Joos and Kim \cite{joos2020rainbow} proved a result that can be viewed as a rainbow version of Dirac's theorem. Aharoni et al. \cite{Aharoni2020} provided a rainbow version of Mantel's theorem. In 2021, Cheng, Wang, and Zhao \cite{Cheng2021} considered rainbow pancyclicity and the existence of rainbow Hamilton paths. Subsequently, Cheng et al. \cite{Cheng2023} used the probabilistic method to prove an asymptotic result for the rainbow version of the Hajnal-Szemer\'{e}di theorem. Recently, Montgomery, M\"{u}yesser, and Pehova \cite{Montgomery2020} obtained some asymptotically tight minimum degree conditions for a family of $n$-graphs to have a rainbow $F$-factor or a rainbow tree with maximum degree $o(\frac{n}{\log n}).$ Li, Li, and Li \cite{Li2023} studied the existence of rainbow spanning trees and rainbow Hamilton paths in a family of graphs under Ore-type conditions. Moreover, they investigated rainbow vertex-pancyclicity and rainbow panconnectedness, as well as the existence of rainbow cliques in a family of graphs under Dirac-type conditions. The study of rainbow structures under spectral conditions has received considerable attention in recent years. Guo et al. \cite{guo2023spectral} gave a sufficient condition in terms of the spectral radius for the existence of a rainbow matching in a family of graphs. He, Li, and Feng \cite{he2024spectral} provided a spectral radius condition for a family of graphs to admit a rainbow Hamilton path. Moreover, they also gave a spectral radius condition on a family of graphs to guarantee a rainbow linear forest of given size. Zhang and van Dam \cite{zhang2025} provided a sufficient condition based on size and spectral radius for the existence of a rainbow Hamilton cycle in a family of graphs, respectively. Zhang and Zhang \cite{zhangzhang2025} gave a sufficient condition in terms of the spectral radius for a family of graphs to contain a rainbow $k$-factor.

The bipartite version of Problem \ref{prob1.1} is also a natural and intriguing problem. Bradshaw\cite{Bradshaw2021} proposed  minimum degree conditions for a family of bipartite graphs to admit a rainbow Hamilton cycle and a rainbow perfect matching, respectively. Furthermore, Bradshaw\cite{Bradshaw2021} proved a stronger result which states that a family of bipartite graphs is bipancyclic under the minimum degree condition. Motivated by \cite{Bradshaw2021}, Hu et al.\cite{Hu2024} investigated the minimum degree condition of vertex-bipancyclicity in a family of bipartite graphs. Chen, Liu, and Yuan\cite{Chen2026} presented tight sufficient conditions in terms of the
spectral radius for a family of bipartite graphs to admit a rainbow Hamilton path and cycle, respectively.

 Based on Theorem \ref{thm1.1} and Problem \ref{prob1.1}, a natural and interesting problem arises.
\begin{prob}\label{prob1.2}
What is the sufficient condition in terms of the spectral radius for the existence of rainbow $k$-factors in a family of bipartite graphs?
\end{prob}
Shi, Li and Chen\cite{Shi2024} provided a sufficient condition in terms of the spectral radius for the existence of a rainbow matching in a family of bipartite graphs.
\begin{thm}[Shi, Li and Chen\cite{Shi2024}]\label{lem2.5}
    Let $\mathcal{G}=\{G_1, G_2, \ldots, G_k\}$ be a family of spanning subgraphs of $K_{n,n}.$ If
    $$\rho(G_i)\geq \sqrt{(k-1)n},$$ for all $i\in[k],$ then $\mathcal{G}$ contains a rainbow matching unless $G_1=G_2=\dots=G_k\cong K_{k-1,n}\cup (n-k+1)K_1.$
\end{thm}
By taking $k=n$ in Theorem \ref{lem2.5}, one can immediately obtain a sufficient condition based on the spectra radius for the existence of a rainbow perfect matching in a family of balanced bipartite graphs, which answers Problem \ref{prob1.2} for $k=1.$
\begin{cor}[Shi, Li and Chen\cite{Shi2024}]\label{cor1.1}
    Let $\mathcal{G}=\{G_1, G_2, \ldots, G_n\}$ be a family of balanced bipartite graphs on vertex set $[2n]$ and the bipartition $(X, Y).$ If
    $$\rho(G_i)\geq \sqrt{(n-1)n},$$ for all $i\in[n],$ then $\mathcal{G}$ contains a rainbow perfect matching unless $G_1=G_2=\dots=G_n\cong K_{n-1,n}\cup K_1.$
\end{cor}
For general $k\geq 2,$ we provide a complete solution to Problem \ref{prob1.2}.
\begin{thm}\label{thm1.2}
    Let $k\geq 2$ be a positive integers and let $\mathcal{G}=\{G_1, G_2, \ldots , G_{kn}\}$ be a family of balanced bipartite graphs on vertex set $[2n]$ and the bipartition $(X, Y),$ where $n\geq 2k.$ If
    \begin{eqnarray*}
        \rho(G_i)\geq \rho(K_{k-1, n-1}\sqcup \widehat{K_{n-k+1, 1}})
    \end{eqnarray*}
     for every $i\in [kn],$ then $\mathcal{G}$ admits a rainbow $k$-factor unless $G_1=G_2=\cdots=G_{kn}\cong K_{k-1, n-1}\sqcup \widehat{K_{n-k+1, 1}}.$
\end{thm}

\section{Preliminaries}
In this section, we first give an important technique and several auxiliary results, all of which will be employed in our subsequent arguments.

In extremal set theory, the shifting technique is one of the most fundamental and widely applied tools. The shifting operation on graphs, also referred to as the Kelmans operation (see, e.g., \cite{brouwer2}), enables us to concentrate on sets with a particular structural property. Let $G$ be a graph on vertex set $[n].$ Define the $(x, y)$-shift $S_{xy}$ as $S_{xy}(G) = \{S_{xy}(e) :e \in E(G)\},$ where
\begin{equation}
\nonumber
S_{xy}(e)= \left\{
\begin{array}{cl}
(e\setminus\{y\})\cup\{x\},         &\text{if $y\in e,$ $x\notin e$ and $(e\setminus\{y\})\cup\{x\}\notin E(G)$}; \\
e,                     & \text{otherwise.}
\end{array} \right.
\end{equation}
If $S_{xy}(G) = G$ for every pair $(x, y)$ satisfying $x < y,$ then $G$ is said to be shifted, denoted by $S(G).$ By the definition of the $(x, y)$-shift, we have $|E(G)|=|E(S_{xy}(G))|.$ In other words, the $(x, y)$-shift preserves the number of edges. Moreover, it is natural to ask whether it affects the spectral radius of a graph. In 2009, Csikv\'{a}ri\cite{csikvari2009conjecture} proved that the shifting operation does not decrease the spectral radius of a graph.
\begin{lemma}[Csikv\'{a}ri\cite{csikvari2009conjecture}]\label{lem2.1}
    Let $x, y$ be two vertices of $G.$ Then $\rho(S_{xy}(G))\geq \rho(G).$
\end{lemma}

Guo et al.\cite{guo2023spectral} determined changes of the spectral radius after $(x, y)$-shift operation in connected graphs.
\begin{lemma}[Guo et al.\cite{guo2023spectral}]\label{lem2.2}
    Let $G$ be a connected graph on vertex set $[n].$ Let $x, y$ be two vertices of $G.$ Then $\rho(S_{xy}(G))> \rho(G)$ unless $G\cong S_{xy}(G).$
\end{lemma}
For a bipartite graph, to ensure that the graph remains bipartite after the $(x, y)$-shift, we must choose the vertices $x$ and $y$ from the same part. Let $G$ be a bipartite graph on vertex set $[n].$ If $S_{xy}(G) = G$ for every pair $(x, y)$ in the same part of $G$ satisfying $x < y,$ then $G$ is said to be bi-shifted. The following observation shows a basic property of $(x, y)$-shift for a bipartite graph $G.$

\begin{observation}\label{obs::1}
Let $G$ be a bipartite graph on vertex set $[n]$ with the partition $(X, Y).$  If $G$ is bi-shifted, then for any $\{x_1, x_2\}\subseteq X$ and $\{y_1, y_2\} \subseteq Y$ such that $x_1 \leq x_2$ and $y_1\leq y_2,$ $\{x_2, y_2\} \in E(G)$ implies $\{x_1, y_1\} \in E(G).$
\end{observation}
Iterating the $(x, y)$-shift over all pairs $(x, y)$ with $x < y$ eventually yields a shifted graph (see \cite{frankl1987} and \cite{frankl1995shifting}). For a bipartite graph $G$, it can be verified that by repeatedly applying the shifting operation in the same part of $G$, one can obtain a bi-shifted graph. Let $S(\mathcal{G})=\{S(G_1), S(G_2), \ldots , S(G_n)\}.$

\begin{lemma}[Zhang and zhang\cite{zhangzhang2025}]\label{lem2.3}
   Let $k$ and $n$ be two positive integers such that $kn$ is even, and let $\mathcal{G}=\{G_1, G_2, \ldots, G_{\frac{kn}{2}}\}$ be a family of graphs on vertex set $[n].$ If $S(\mathcal{G})$ admits a rainbow $k$-factor, then so does $\mathcal{G}.$
\end{lemma}

Let $A$ be a symmetric real matrix whose rows and columns are indexed by $X=[n].$ Given a partition $\Pi:$ $X=X_1\cup X_2\cup \dots \cup X_m,$ the matrix $A$ is denoted by
\begin{eqnarray*}
A=\begin{bmatrix}
   A_{11}&\dots&A_{1m}\\\vdots &\ddots&\vdots\\A_{m1}&\dots&A_{mm}
\end{bmatrix},
\end{eqnarray*}
where $A_{ij}$ is the submatrix of $A$ with respect to rows in $X_i$ and columns in $X_j.$ Let $B_{\Pi}$ be a matrix of order $m$ whose $(i,j)$-entry equals the average row sum of $A_{ij}.$ Then $B_{\Pi}$ is called a quotient matrix of $A$ corresponding to this partition. If the row sum of each block $A_{ij}$ is constant, then the partition $\Pi$ is equitable.
\begin{lemma}[Brouwer and Haemers\cite{brouwer2}, Godsil and Royle\cite{Godsil}]\label{lem2.4} Let $A$ be a real symmetric matrix and $\rho(A)$ be its largest eigenvalue. Let $B_{\Pi}$ be an equitable quotient matrix of $A$. Then the eigenvalues of $B_{\Pi}$ are also eigenvalues of $A$. Furthermore, if $A$ is nonnegative and irreducible, then $\rho(A)=\rho(B_{\Pi}).$
\end{lemma}
\section{Proof of Theorem \ref{thm1.2}}
For convenience, let $B_{n,k}=K_{k-1, n-1}\sqcup \widehat{K_{n-k+1, 1}}.$ Before presenting our proof, we show some crucial lemmas.
\begin{lemma}\label{lem3.1}
Let $k\geq 2$ be a positive integer and let $G$ be a balanced bipartite graph on vertex set $[2n]$ and the bipartition $(X, Y),$ where $n\geq 2k.$ Let $\{u, v\} \subseteq X$ or $\{u, v\} \subseteq Y$ with $u<v.$ If $\rho (G) \geq \rho(B_{n,k})$ and $S_{uv}(G)\cong B_{n,k},$ then $G\cong B_{n,k}.$
\end{lemma}
\begin{proof}
Note that $S_{uv}(G)\cong B_{n,k}.$ Let $w$ be the vertex of degree $k-1$ in $S_{uv}(G).$ Without loss of generality, we assume that $w\in X.$ Denote $X_1=X\setminus\{w\}.$ If $\{u, v\} \subseteq X_1,$ then by the definition of $(u, v)$-shift, we have $G=S_{uv}(G)\cong B_{n,k}.$ If $u=w$ or $v=w,$ then it follows from $u<v$ that $u\in X_1$ and $v=w\in X.$ We claim that $d_G(u)>0$ and $d_G(v)>0.$ In fact, if $d_G(u)=0$ or $d_G(v)=0,$ then $d_{S_{uv}(G)}(u)=0$ or $d_{S_{uv}(G)}(v)=0,$ contradicting $S_{uv}(G)\cong B_{n,k}.$ Then $d_G(u)>0$ and $d_G(v)>0,$ and hence $G$ is connected. Note that
\begin{eqnarray*}
\rho(G)\geq \rho(B_{n,k})=\rho(S_{uv}(G)).
\end{eqnarray*}
By Lemma \ref{lem2.2}, we have $\rho(G)=\rho(S_{uv}(G))$ and $G\cong S_{uv}(G)\cong B_{n,k}.$ If $\{u, v\} \subseteq Y,$ then we can obtain that $G= S_{uv}(G)\cong B_{n,k}.$
\end{proof}
By Lemmas \ref{lem2.1} and \ref{lem3.1}, we immediately obtain the following result.
\begin{cor}\label{cor3.1}
Let $k\geq 2$ be a positive integer and let $G$ be a balanced bipartite graph on vertex set $[2n]$ and the bipartition $(X, Y),$ where $n\geq 2k.$ If $\rho(G) \geq \rho(B_{n,k})$ and $S(G)\cong B_{n,k},$ then $G\cong B_{n,k}.$
\end{cor}
Next we prove a technical lemma which is very important to our main result.
\begin{lemma}\label{lem3.2}
   Let $ k\geq 2$ be a positive integer and let $\mathcal{G}=\{G_1, G_2, \ldots , G_{kn}\}$ be a family of balanced bipartite graphs on vertex set $[2n]$ and the bipartition $(X, Y),$ where $n\geq 2k.$ If $G_i\cong B_{n,k}$ for every $i\in [kn],$ and there exist $t_1, t_2\in [kn]$ such that $G_{t_1} \neq G_{t_2},$ then $\mathcal{G}$ admits a rainbow $k$-factor.
\end{lemma}
\begin{proof}
Since $G_i\cong B_{n,k},$ there exists a vertex of degree $k-1,$ and every other vertex has degree either $n-1$ or $n$ in $G_i$ for each $i\in [kn].$ Suppose that $G_1, G_2, \ldots , G_{kn}$ have $p$ distinct vertices with degree $k-1,$ denoted by $u_1,..., u_p.$ Now we distinguish our proof into the following two cases.
\begin{case}\label{case1}
    $p=1.$
\end{case}
Without loss of generality, we assume that $G_1\neq G_2.$ Then $|\cup_{i=1}^k N_{G_i}(u_1)|\geq k.$ Therefore, we can choose $k$ distinct vertices $v_1, v_2,\dots, v_k$ such that $v_i\in N_{G_i}(u_1)$ for $1\leq i\leq k.$ It follows that $u_1v_i\in E(G_i)$ for $1\leq i\leq k.$ Next we construct a rainbow $k$-factor of $\mathcal{G}.$ Let $G'_i$ be the induced subgraph of $G_i$ on $[2n]\setminus \{u_1\}.$ Then $G'_{k+1}=G'_{k+2}=\cdots=G'_{kn}\cong K_{n.n-1}.$ Hence we can obtain a rainbow bipartite spanning subgraph $F_0$ of $\{G'_{k+1}, G'_{k+2}, \ldots , G'_{kn}\}$ such that
\begin{eqnarray*}
    d_{F}(u) = \left\{ \begin{array}{cc}
        k-1, & u=v_i\text{ for }1\leq i\leq k, \\
        k, & \text{others.} \\
        \end{array} \right.
\end{eqnarray*}
Note that $F_0$ is also a rainbow subgraph of $\{G_{k+1}, G_{k+2}, \ldots , G_{kn}\}$ and $u_1v_i\in E(G_i)$ for $1\leq i\leq k.$ By adding the vertex $u_1$ and edges $u_1v_1, u_1v_2, \dots u_1v_k$ to $F_0,$ we obtain a rainbow $k$-factor $F$ of $\mathcal{G}.$
\begin{case}
    $p\geq 2.$
\end{case}
\setcounter{case}{0}
Recall that $u_1,\dots, u_p$ are the $p$ distinct vertices with degree $k-1$ in $\mathcal{G}.$ Let
\begin{eqnarray*}
    \mathcal{G}_i=\{G_j: d_{G_j}(u_i)=k-1\}
\end{eqnarray*}
and $|\mathcal{G}_i|=n_i$ for $1\leq i\leq p,$ where $\sum_{i=1}^{p}n_i=kn.$ Without loss of generality, we may assume that $n_1\geq n_2\geq \cdots \geq n_p.$ Note that $p\geq 2.$ Then $n_1< kn.$ Let $k_i=\lfloor \frac{n_i}{n} \rfloor$ for $1\leq i\leq p.$ Clearly, $0\leq k_p\leq k_{p-1}\leq \cdots \leq k_1<k.$ If $k_1>0,$ then we choose the largest integer $q\leq p$ such that $k_q>0,$ and let $k'=\sum_{i=1}^q k_i\leq k.$ If $k_1=0,$ then let $q=0$ and $k'=0.$ Choose a subset $\mathcal{G}'_i$ of $\mathcal{G}_i$ such that $|\mathcal{G}'_i|=k_in$ for $1\leq i\leq q.$ Assume that $\mathcal{G}'_i=\{G_1^{(i)}, G_2^{(i)},\dots , G_{k_in}^{(i)}\}.$ Let $\mathcal{G}'=\cup_{i=1}^q\mathcal{G}'_i.$ We first prove the following claim.
\begin{claim}\label{claim1}
    $\mathcal{G}'$ has a rainbow $k'$-factor.
\end{claim}
\begin{proof}
If $q=0,$ then the result is trivial. Assume that $q\geq 1.$ Note that $k_i\leq k-1.$ Then $|\cup_{j=1}^{k_i} N_{G^{(i)}_j}(u_i)|\geq k-1\geq k_i.$ Therefore, we can choose $k_i$ distinct vertices $v_1, v_2,\dots, v_{k_i}$ such that $v_j\in N_{G^{(i)}_j}(u_i)$ for $1\leq j\leq k_i.$ It follows that $u_1v_j\in E(G^{(i)}_j)$ for $1\leq j\leq k_i.$ By using the same argument as the Case \ref{case1}, we can obtain that each $\mathcal{G}'_i$ has a rainbow $k_i$-factor $F_i,$ where $E(F_i)=\{e^{(i)}_1, e^{(i)}_2,\dots , e^{(i)}_{k_in}\}$ and $e^{(i)}_j\in E(G^{(i)}_j)$ for $1\leq i\leq q$ and $1\leq j\leq k_in.$

If $q=1,$ then $k'=k_1$ and $\mathcal{G}'=\mathcal{G}'_1.$ Hence $F_1$ is a rainbow $k'$-factor of $\mathcal{G}'.$ Next we consider $q\geq 2.$ Let $F=F_1\cup F_2\cup \cdots \cup F_q.$ Note that $d_F(u)=k'$ for any $u\in V(F).$ If $E(F_r)\cap E(F_s)=\emptyset$ for any $1\leq r<s\leq q,$ then $F$ is a rainbow $k'$-factor of $\mathcal{G}'.$ If there exist some $1\leq r<s\leq q$ such that $E(F_r)\cap E(F_s)\neq \emptyset,$ then $F$ has multiple edges. Assume that $e^{(r)}_{i_1}=e^{(s)}_{i_2}=vv'\in E(F_r)\cap E(F_s),$ where $v\in X$ and $v'\in Y.$ For the rainbow subgraph $F,$ we next present a multiple-edge removal method that guarantees $d_F(u)=k'$ for every vertex $u\in [2n].$

Note that $q\geq 2.$ Then $k'\geq 2$ and hence $n\geq 2k\geq 2k'\geq 4.$ Then
\begin{eqnarray*}
    &&|E(F)|-2(k'-2)-2(k'-2)(k'-1)-2\\
    &=&k'n-2(k'-2)k'-2\\
    &\geq&2k'^2-2(k'-2)k'-2\\
    &>&1,
\end{eqnarray*}
and hence we can choose an edge $ww'=e^{(t)}_{i_3}\in E(F_t)$ such that $vw',v'w\notin E(F),$ where $w\in X$ and $w'\in Y.$ Without loss of generality, assume that $t\neq r.$ Note that at least one of $v$ and $v'$ is distinct from $u_r,$ and the same holds for $u_t.$ We assume that $v\neq u_t$ and $v'\neq u_r.$

If $vw'\in E(G^{(t)}_{i_3})$ and $v'w\in E(G^{(r)}_{i_1}),$ then we can replace $vv'\in E(G^{(r)}_{i_1})$ by $v'w\in E(G^{(r)}_{i_1})$ and $ww'\in E(G^{(t)}_{i_3})$ by $vw'\in E(G^{(t)}_{i_3}).$ Otherwise, $vw'\notin E(G^{(t)}_{i_3})$ or $v'w\notin E(G^{(r)}_{i_1}).$ By symmetry, we only need to consider the case $v'w\notin E(G^{(r)}_{i_1}).$ Therefore, $v'=u_r$ or $w=u_r.$ Since $v'\neq u_r,$ we have $w=u_r$ and $vw'\in E(G^{(r)}_{i_1}).$ If $v'w\in E(G^{(t)}_{i_3}),$ then replace $vv'\in E(G^{(r)}_{i_1})$ by $vw'\in E(G^{(r)}_{i_1})$ and $ww'\in E(G^{(t)}_{i_3})$ by $v'w\in E(G^{(t)}_{i_3}).$
If $v'w\notin E(G^{(t)}_{i_3}),$ then $v'=u_t$ or $w=u_t.$ Since $w=u_r\neq u_t,$ we have $v'=u_t.$ Note that
\begin{eqnarray*}
    &&|E(F)|-2(k'-2)-2(k'-2)(k'-1)-2(k'-1)-3\\
    &=&k'n-2(k'-2)(k'+1)-5\\
    &\geq& 2k'^2-2(k'-2)(k'+1)-5\\
    &\geq&1.
\end{eqnarray*}
Then there exists an edge $zz'=e^{(t')}_{i_4}\in E(F_{t'})$ such that $z\neq w, z'\neq w'$ and $vz',v'z\notin E(F),$ where $z\in X$ and $z'\in Y.$ This implies that $v'z\in E(G^{(r)}_{i_1}).$ If $vz'\in E(G^{(t')}_{i_4}),$ then replace $vv'\in E(G^{(r)}_{i_1})$ by $v'z\in E(G^{(r)}_{i_1})$ and $zz'\in E(G^{(t')}_{i_4})$ by $vz'\in E(G^{(t')}_{i_4}).$ If $vz'\notin E(G^{(t')}_{i_4}),$ then $v=u_{t'}$ or $z'=u_{t'},$ which implies that $vz'\in E(G^{(t)}_{i_3})$ and $ww'\in E(G^{(t')}_{i_4}).$ Then replace $vv'\in E(G^{(r)}_{i_1})$ by $v'z\in E(G^{(r)}_{i_1}),$ $zz'\in E(G^{(t')}_{i_4})$ by $ww'\in E(G^{(t')}_{i_4})$ and $ww'\in E(G^{(t)}_{i_3})$ by $vz'\in E(G^{(t)}_{i_3}).$

Repeating the above steps until $E(F_r)\cap E(F_s)=\emptyset$ for any $1\leq r<s\leq q,$ we can obtain a rainbow $k'$-factor of $\mathcal{G}'.$
\end{proof}
\setcounter{case}{0}
By Claim \ref{claim1}, we can obtain a rainbow $k'$-factor $F'$ of $\mathcal{G}'$ such that $E(F')=\cup_{i=1}^q\{e^{(i)}_1, e^{(i)}_2,\dots , e^{(i)}_{k_in}\},$ where $e^{(i)}_j\in E(G^{(i)}_j)$ and $G^{(i)}_j\in \mathcal{G}'_i$ for $1\leq i\leq q$ and $1\leq j\leq k_in.$

If $k'=k,$ then $\mathcal{G}'=\mathcal{G}.$ Hence $\mathcal{G}$ has a rainbow $k$-factor $F',$ and the lemma follows. Next we consider the case $0\leq k'\leq k-1.$ Without loss of generality, we suppose that $\mathcal{G}''=\mathcal{G}\setminus\mathcal{G}'=\{G_{k'n+1},..., G_{kn} \}.$
Let $G'_i\cong K_{n,n-1}\cup K_1$ be the subgraph of $G_i$ for $k'n+1\leq i\leq kn.$ Moreover, if $d_{G_i}(u)=k-1,$ then we have $d_{G'_i}(u)=0.$ Now we divide $\mathcal{G}''$ into $k-k'$ groups. Our next goal is to prove that each group has a rainbow perfect matching.
\begin{claim}\label{claim2}
   For $k'+1\leq i\leq k,$ $\{G'_{(i-1)n+1}, G'_{(i-1)n+2},..., G'_{in} \}$ has a rainbow perfect matching.
\end{claim}
\begin{proof}
    According to the definition of $\mathcal{G}'_i,$ we can obtain that for any $n$ graphs of $\{G'_{k'n+1}, G'_{k'n+2}, ..., G'_{kn} \},$ there must exist at least two graphs, say $G'$ and $G'',$ such that $G'\neq G''.$ Note that $G'_{(i-1)n+j}\cong K_{n,n-1}\cup K_1$ for $k'+1\leq i\leq k$ and $1\leq j\leq n.$ Then $\rho(G'_{(i-1)n+j})=\sqrt{n(n-1)}.$ By Corollary \ref{cor1.1},  we can obtain a rainbow perfect matching of $\{G'_{(i-1)n+1}, G'_{(i-1)n+2},..., G'_{in} \}.$
\end{proof}
By Claim \ref{claim2}, $\{G'_{(i-1)n+1}, G'_{(i-1)n+2}, ..., G'_{in}\}$ has a rainbow perfect matching, denoted by $M_i,$ where $k'+1\leq i\leq k.$ Clearly, $M_i$ is also a rainbow perfect matching of $\{G_{(i-1)n+1}, G_{(i-1)n+2}, ..., G_{in}\}.$ Let $E(M_i)=\{e^{(i)}_1, e^{(i)}_2,\dots , e^{(i)}_{n}\}$ such that $e^{(i)}_j\in E(G_{(i-1)n+j})$ for $k'+1\leq i\leq k$ and $1\leq j\leq n.$ Then for any edge $e^{(i)}_j=u'u''\in M_i,$ we always have $u',u''\neq u,$ where $d_{G_{(i-1)n+j}}(u)=k-1.$ Let $F=F'\cup M_{k'+1}\cup M_{k'+2}\cup \cdots \cup M_k.$ Note that $d_F(u)=k$ for any $u\in V(F).$ If $F$ has no multiple edges, then $F$ is a rainbow $k$-factor of $\mathcal{G}.$ Next we consider the case that $F$ has multiple edges. Suppose that $e^{(r)}_{i_1}=e^{(s)}_{i_2}=vv'\in E(F),$ where $k'+1\leq r<s\leq k$ or $1\leq r\leq q< k'+1\leq s\leq k.$ Then $e^{(s)}_{i_2}\in E(G_{(s-1)n+i_2}).$ Assume that  $v\in X,$ $v'\in Y$ and $G_{(s-1)n+i_2}\in \mathcal{G}_l,$ where $1\leq l\leq p.$ Then $v,v'\neq u_l.$ Next we present a multiple-edge removal method which guarantees $d_F(u)=k$ for any $u\in V(F).$

Note that $k\geq 2.$ Then
\begin{eqnarray*}
    &&|E(F)|-2(k-2)-2(k-2)(k-1)-k-2\\
    &=&kn-2(k-1)k-2\\
    &\geq&2k^2-2(k-1)k-2\\
    &>&1.
\end{eqnarray*}
Then there exist an edge $ww'=e^{(t)}_{i_3}\in E(F)$ such that $vw',v'w\notin E(F)$ and $w,w'\neq u_l,$ where $w\in X$ and $w'\in Y.$ If $e^{(t)}_{i_3}\in E(F'),$ then $e^{(t)}_{i_3}\in E(G^{(t)}_{i_3}).$ If $e^{(t)}_{i_3}\in E(M_{k'+1}\cup M_{k'+2}\cup \cdots \cup M_{k}),$ then $e^{(t)}_{i_3}\in E(G_{(t-1)n+i_3}).$ Without loss of generality, we assume that $G_{(t-1)n+i_3}\in \mathcal{G}_t.$ Recall that $v,v',w,w'\neq u_l.$ Then $vw', v'w\in E(G_{(s-1)n+i_2}).$ We claim that $vw'\in E(H)$ or $ v'w\in E(H),$ where $H\in\{G^{(t)}_{i_3}, G_{(t-1)n+i_3}\}.$ In fact, if $vw'\notin H,$ then $v=u_t$ or $w'=u_t,$ and hence $v'w\in H.$ Similarly, if $v'w\notin H,$ then $vw'\in H.$ Then replace $vv'\in E(G_{(s-1)n+i_2})$ by $vw'\in E(G_{(s-1)n+i_2})$ and $ww'\in E(H)$ by $v'w\in E(H)$ or $vv'\in E(G_{(s-1)n+i_2})$ by $v'w\in E(G_{(s-1)n+i_2})$ and $ww'\in E(H)$ by $vw'\in E(H).$

Repeating the above steps until $F$ has no multiple edges, we can obtain a rainbow $k$-factor of $\mathcal{G}.$
\end{proof}
\begin{lemma}\label{lem3.3}
    Let $k\geq 2,$ $n\geq 2k$ and $k+1\leq p\leq n-1$ be three positive integers. Then
$$\rho(K_{p-1, n+k-p+1}\sqcup \widehat{K_{n-p+1, p-k+1}})<\rho(B_{n,k}).$$
\end{lemma}
\begin{proof}
Let $G=B_{n,k}$ and $G'=K_{p-1, n+k-p+1}\sqcup \widehat{K_{n-p+1, p-k+1}}.$ The adjacent matrix $A(G)$ of $G$ has the equitable quotient matrix
\begin{eqnarray*}
B_{\Pi_1}=\begin{bmatrix}
   0&0&n-1&1\\
   0&0&n-1&0\\
   k-1&n-k+1&0&0\\
   k-1&0&0&0
\end{bmatrix}.
\end{eqnarray*}
According to Lemma \ref{lem2.4}, we have $\rho(G)=\rho(B_{\Pi_1}).$ By simple computation, the characteristic polynomial of $B_{\Pi_1}$ is
\begin{eqnarray*}
P_1(x)=x^4-[n(n-1)+(k-1)]x^2+(n-1)(n-k+1)(k-1).
\end{eqnarray*}
The adjacent matrix $A(G')$ of $G'$ has the following equitable quotient matrix
\begin{eqnarray*}
B_{\Pi_2}=\begin{bmatrix}
   0&0&n+k-p-1&p-k+1\\
   0&0&n+k-p-1&0\\
   p-1&n-p+1&0&0\\
   p-1&0&0&0
\end{bmatrix}.
\end{eqnarray*}
By Lemma \ref{lem2.4}, we obtain that $\rho(G')=\rho(B_{\Pi_2}).$ Note that the characteristic polynomial of $B_{\Pi_2}$ is
\begin{eqnarray*}
P_2(x)&=&x^4-[n(n+k-p-1)+(p-1)(p-k+1)]x^2\\
&&+(n+k-p-1)(p-k+1)(n-p+1)(p-1).
\end{eqnarray*}
Let $P(x)=P_1(x)-P_2(x).$ Then
\begin{eqnarray*}
    P(x)&=&-(n-p)(p-k)x^2+(n-1)(n-k+1)(k-1)\\
    &&-(n+k-p-1)(p-k+1)(n-p+1)(p-1).
\end{eqnarray*}
Note that $K_{n,n-1}$ is a proper subgraph of $G.$ Then $\rho(G)>\rho(K_{n,n-1})=\sqrt{n(n-1)}.$ Moreover, we can obtain that
\begin{eqnarray*}
    P(\sqrt{n(n-1)})&=&-n(n-p)(p-k)(n-1)+(n-1)(n-k+1)(k-1)\\
    &&-(n+k-p-1)(p-k+1)(n-p+1)(p-1)\\
    &=&(k-p)(n-p)(-p^2+(k+n)p+n(n-1)-2k-2).
\end{eqnarray*}
Since $k+1\leq \frac{k+n}{2}\leq n-1$ and $k\geq 2,$ we have
\begin{eqnarray*}
    -p^2+(k+n)p+n(n-1)-2k-2&\geq&-(n-1)^2+(k+n)(n-1)+n(n-1)-2k-2\\
                           &=&n^2+kn-3k+1\\
                           &\geq& 4k^2+2k^2-3k+1\\
                           &=&3k(k-2)+1\\
                           &>&0.
\end{eqnarray*}
By $k+1\leq p\leq n-1,$ we can obtain
\begin{eqnarray*}
    P(\sqrt{n(n-1)})=(k-p)(n-p)(-p^2+(k+n)p+n(n-1)-2k-2)<0.
\end{eqnarray*}
It is straightforward to check that if $x>\sqrt{n(n-1)},$ then $P(x)<P(\sqrt{n(n-1)})<0,$ which implies that $\rho(G')<\rho(G).$
\end{proof}

Now we are in a position to present the proof of Theorem \ref{thm1.2}.
\vspace{1mm}

\medskip
\noindent  \textbf{Proof of Theorem \ref{thm1.2}.}
 Suppose to the contrary that $\mathcal{G}=\{G_1, G_2, \ldots,$ $ G_{kn}\}$ has no rainbow $k$-factor. For convenience, we denote $S_i=S(G_i)$ for every $ i\in [kn].$ It follows from Lemma \ref{lem2.3} that the bi-shifted family $\{S_1,$ $S_2,$ $\ldots,$ $S_{kn}\}$ has no rainbow $k$-factor. By Lemma \ref{lem2.1} and $\rho(G_i)\geq \rho(B_{n,k})$, we have
\begin{eqnarray}\label{eq1}
    \rho(S_i)\geq \rho(G_i)\geq \rho(B_{n,k}).
\end{eqnarray}
Next we show that $S_i\cong B_{n,k}$ for every $i\in[kn].$ Without loss of generality, we suppose that $X=[n]$ and $Y=[2n]\setminus[n].$ Define $e_j=\{j, 2n+k-j\},$ where $k\leq j\leq n.$
\begin{claim}\label{claim3}
$\{e_{k+1}, e_{k+2}, \ldots, e_{n-1}\}\subseteq E(S_i)$ for each $i\in [kn].$
\end{claim}
\begin{proof}
Assume that Claim \ref{claim3} does not hold. Then there exist some $k+1\leq p \leq n-1$ and some $q\in [kn]$ such that $e_p=\{p, 2n+k-p\} \notin E(S_q).$ Since $S_q$ is bi-shifted, by Observation \ref{obs::1}, we can obtain that $1\leq i<p$ or $n+1\leq j <2n+k-p$ for each edge $\{i, j\}\in E(S_q).$ This implies that $S_q$ is a subgraph of $K_{p-1, n+k-p+1}\sqcup \widehat{K_{n-p+1, p-k+1}}.$ Hence $\rho(S_q)\leq \rho(K_{p-1, n+k-p+1}\sqcup \widehat{K_{n-p+1, p-k+1}}).$ By Lemma \ref{lem3.3}, we obtain that
\begin{eqnarray*}
    \rho(S_q)\leq \rho(K_{p-1, n+k-p+1}\sqcup \widehat{K_{n-p+1, p-k+1}})<\rho(B_{n,k}),
\end{eqnarray*}
which contradicts (\ref{eq1}).
\end{proof}
\begin{claim}\label{claim4}
    For $i\in[kn],$ $\delta(S_i)\geq k-1.$
\end{claim}
\begin{proof}
    If $\delta(S_i)< k-1,$ then $S_i$ is a proper subgraph of $B_{n,k}.$ This implies that
    $$\rho(S_i)<\rho(B_{n,k}),$$
    which contradicts (\ref{eq1}).
\end{proof}
By Claim \ref{claim4}, we have $\delta(S_i)\geq k-1.$ If $d_{S_i}(2n)=k-1$ for each $i\in[kn],$ then $S_i$ is a subgraph of $B_{n,k}.$ It follows from (\ref{eq1}) that $\rho(S_i)=\rho(B_{n,k}).$ Since $B_{n,k}$ is connected, we have $S_i\cong B_{n,k}$ for each $i\in[kn].$ Similarly, if $d_{S_i}(n)=k-1$ for each $i\in[kn],$ then we have $S_i\cong B_{n,k}.$ Next we assume that there exist $p, q\in [kn]$ such that $d_{s_p}(n)\geq k$ and $d_{s_q}(2n)\geq k.$ Recall that $S_p$ and $S_q$ are bi-shifted. By Observation \ref{obs::1}, we can obtain that $e_n=\{n, n+k\}\in E(S_p)$ and $e_k=\{k, 2n\}\in E(S_q).$
\begin{claim}\label{claim5}
    $p\neq q.$
\end{claim}
\begin{proof}
    Suppose to the contrary that $p=q.$ Without loss of generality, we assume that $p=q=1.$ We first prove that $d_{S_i}(n)=k-1$ and $d_{S_i}(2n)=k-1$ for each $2\leq i\leq kn.$ In fact, if there exists some $2\leq j\leq kn$ such that $d_{S_j}(n)\geq k$ or $d_{S_j}(2n)\geq k,$ then we obtain that $d_{S_1}(n)\geq k$ and $d_{S_j}(2n)\geq k$, or $d_{S_1}(2n)\geq k$ and $d_{S_j}(n)\geq k.$ This implies that $e_n\in E(S_1)$ and $e_k\in E(S_j)$ or $e_k\in E(S_1)$ and $e_n\in E(S_j),$ which contradicts $p=q.$  Hence $S_i$ is a proper subgraph of $B_{n,k}$ for $2\leq i\leq kn.$ Since $B_{n,k}$ is connected, we have $\rho(S_i)<\rho(B_{n,k}),$ which contradicts (\ref{eq1}).
\end{proof}
By Claim \ref{claim5}, we have $p\neq q.$ Without loss of generality, we assume that $p=kn$ and $q=(k-1)n+k.$ Then $e_n\in E(S_{kn})$ and $e_k\in E(S_{(k-1)n+k}).$ Now we construct $k$ perfect matchings. Let $M_i$ be a perfect matching on the vertex set $[2n],$ where $1\leq i\leq k.$ Denote $E(M_i)=\{e^{(i)}_1, e^{(i)}_2, \dots , e^{(i)}_n\},$ where
\begin{eqnarray*}
    e^{(i)}_j = \left\{ \begin{array}{cc}
        \{j, n+i-j\}, & 1\leq j\leq i-1; \\
        \{j, 2n+i-j\}, & i\leq j\leq n.
        \end{array} \right.
\end{eqnarray*}
\begin{claim}\label{claim6}
    $e^{(i)}_j\in E(S_t)$ for each $t\in[kn],$ where $e^{(i)}_j\neq e_k, e_n.$
\end{claim}
\begin{proof}
    By $\delta(S_t)\geq k-1$ and Observation \ref{obs::1}, we have $\{i,j\}\in E(S_t)$ for each $t\in[kn],$ where $1\leq i\leq k-1$ and $n+1\leq j\leq 2n.$ Then for $1\leq i\leq k$ and $1\leq j\leq i-1,$ $\{j, n+i-j\}\in E(S_t)$ and for $1\leq i\leq k-1$ and $i\leq j\leq k-1,$ $\{j, 2n+i-j\}\in E(S_t).$ By Claim \ref{claim3}, $e_j=\{j, 2n+k-j\}\in E(S_t)$ for $k+1\leq j\leq n-1.$ According to Observation \ref{obs::1} we have $\{j, 2n+i-j\}\in E(S_t),$ where $1\leq i\leq k-1$ and $j=k$ or $1\leq i\leq k$ and $k+1\leq j\leq n-1.$ Hence we can obtain that $e^{(i)}_j\in E(S_t)$ for each $t\in[kn],$ where $e^{(i)}_j\neq e_k, e_n.$
\end{proof}
Note that $e^{(k)}_k=e_k$ and $e^{(k)}_n=e_n.$ Combining Claim \ref{claim6}, $e_k\in E(S_{(k-1)n+k})$ and $e_n\in E(S_{kn}),$ we choose $e^{(i)}_j\in E(S_{(i-1)n+j})$ for $1\leq i\leq k$ and $1\leq j\leq n.$ Then $M_i$ is a rainbow perfect matching of $\{S_{(i-1)n+1},$ $S_{(i-1)n+2},$ $\ldots,$ $S_{in}\}.$ Let $F=M_1\cup M_2\cup \cdots \cup M_k.$ Since $E(M_i)\cap E(M_j)=\emptyset$ for $1\leq i<j\leq k,$ $F$ is a rainbow $k$-factor of $\mathcal{G},$ a contradiction.

Therefore, $S_i\cong B_{n,k}$ for each $i\in [kn].$ Combining (\ref{eq1}), we have $\rho(S_i)=\rho(G_i)=\rho(B_{n,k}).$ By Corollary \ref{cor3.1}, $G_i\cong B_{n,k}.$ If there exist $p,q\in [kn]$ such that $G_p\neq G_q,$ then by Lemma \ref{lem3.2}, $\mathcal{G}$ has a rainbow $k$-factor, a contradiction. Hence $G_1=G_2=\cdots=G_{kn}$ and $G_1\cong B_{n,k}.$
\hspace*{\fill}$\Box$

\vspace{5mm}
\noindent
{\bf Declaration of competing interest}
\vspace{3mm}

The authors declare that they have no known competing financial interests or personal relationships that could have appeared to influence the work reported in this paper.

\vspace{5mm}
\noindent
{\bf Data availability}
\vspace{3mm}

No data was used for the research described in this paper.

\vspace{5mm}
\noindent
{\bf Acknowledgement}

\vspace{3mm}
The research of Ruifang Liu is supported by the National Natural Science Foundation of China {(No. 12371361)} and Distinguished Youth Foundation of Henan Province {(No. 242300421045)}.


\end{document}